\documentclass[11pt]{article}
\author{Peter J. Brockwell \thanks{Department of Statistics Colorado State University,
USA,
email: pjbrock@stat.colostate.edu }
\and
Vincenzo Ferrazzano
\thanks{Center for Mathematical Sciences, Technische Universit\"at M\"unchen,  85748 Garching b. M\"unchen, Germany,
email: ferrazzano@ma.tum.de, http://www-m4.ma.tum.de/pers/ferrazzano/}
\and
Claudia Kl\"uppelberg\thanks{Center for Mathematical Sciences, and Institute for Advanced Study, Technische Universit\"at M\"unchen,  85748 Garching b. M\"unchen, Germany,
email: cklu@ma.tum.de, http://www-m4.ma.tum.de}
}

\title{High frequency sampling of a continuous-time ARMA process}
\usepackage{natbib}
\usepackage[latin9]{inputenc}
\usepackage[english]{babel}
\usepackage[T1]{fontenc}
\usepackage{amsmath,amssymb}
\usepackage{lscape}
\usepackage[final]{showkeys}
\usepackage{enumitem}
\usepackage{graphicx}
\usepackage{url}
\usepackage{latexsym}
\usepackage[usenames]{color}
\usepackage{float}
\usepackage{epsfig}
\usepackage{amssymb}
\usepackage{amsfonts}
\usepackage{amsmath}
\usepackage{algorithmic}
\usepackage{amsthm}

\numberwithin{equation}{section}

\newtheorem{thm}{Theorem}[section]
\newtheorem{cor}[thm]{Corollary}
\newtheorem{lem}[thm]{Lemma}
\newtheorem{prop}[thm]{Proposition}
\newtheorem{defn}[thm]{Definition}
\newtheorem{example}[thm]{Example}
\newtheorem{oss}[thm]{Remark}
\newtheorem{fig}[thm]{Figure}

\newcommand{\displayfrac}[2]{\frac{\displaystyle #1}{\displaystyle #2}}

\newcommand{\bthe}{\begin{thm}}
\newcommand{\ethe}{\end{thm}}

\newcommand{\ble}{\begin{lem}}
\newcommand{\ele}{\end{lem}}

\newcommand{\bde}{\begin{defn}}
\newcommand{\ede}{\end{defn}}

\newcommand{\bco}{\begin{cor}}
\newcommand{\eco}{\end{cor}}

\newcommand{\bpr}{\begin{prop}}
\newcommand{\epr}{\end{prop}}

\newcommand{\bproof}{\begin{proof}}
\newcommand{\eproof}{\end{proof}}

\newcommand{\bexam}{\begin{example}\rm}
\newcommand{\eexam}{\halmos\end{example}}

\newcommand{\brem}{\begin{oss}\rm}
\newcommand{\erem}{\halmos\end{oss}}

\newcommand{\bfi}{\begin{fig}}
\newcommand{\efi}{\end{fig}}

\newcommand{\beao}{\begin{eqnarray*}}
\newcommand{\eeao}{\end{eqnarray*}\noindent}

\newcommand{\beam}{\begin{eqnarray}}
\newcommand{\eeam}{\end{eqnarray}\noindent}

\newcommand{\barr}{\begin{array}}
\newcommand{\earr}{\end{array}}

\newcommand{\beq}{\begin{equation}}
\newcommand{\eeq}{\end{equation}}

\def\bbr{{\Bbb R}}
\def\bbn{{\Bbb N}}

\def\bbz{{\Bbb Z}}
\def\bbe{{\Bbb E}}

\newcommand{\dsum}{\displaystyle\sum}

\newcommand{\la}{{\lambda}}

\newcommand{\ga}{{\gamma}}

\newcommand{\si}{{\sigma}}

\newcommand{\CARMA}{{\rm CARMA}}
\newcommand{\CAR}{{\rm CAR}}

\def\halmos{\hfill $\Box$  \medskip }

\begin{document}
\maketitle

\begin{abstract}
Continuous-time autoregressive moving average (CARMA) processes have recently been used
widely in the modeling of non-uniformly spaced data and as a tool for dealing with high-frequency
data of the form $Y_{n\Delta}, n=0,1,2,\ldots$, where $\Delta$ is small and positive.  Such data
occur in many fields of application, particularly in finance and the study of turbulence.  This
paper is concerned with the characteristics of the process $(Y_{n\Delta})_{n\in\bbz}$, when $\Delta$ is small
and the underlying continuous-time process $(Y_t)_{t\in\bbr}$ is a specified CARMA process.
\end{abstract}

\bigskip

\noindent
\begin{tabbing}
\end{tabbing}

\medskip


\bigskip

\section{Introduction}

Throughout this paper we shall be concerned with a CARMA process driven by a second-order zero-mean L\'evy
process $L$ with $EL_1=0$ and $EL_1^2=\sigma^2$.  The process is defined as follows.

For non-negative integers $p$ and $q$ such that $q<p$, a $\CARMA(p,q)$ process $Y=(Y_t)_{{t\in\bbr}}$, with
coefficients
$a_{1},\ldots,a_{p}$, $b_{0},\ldots,b_{q}\in\bbr$, and driving L\'evy process $L$,  is defined to be a
strictly stationary solution of the
suitably interpreted formal equation,
\begin{equation}\label{1.1}
a(D)Y_t=b(D)DL_t,\quad t\in\bbr,
\end{equation}
where $D$ denotes differentiation with respect to $t$, $a(\cdot)$ and $b(\cdot)$ are the polynomials,
$$a(z):=z^{p}+a_{1}z^{p-1}+\cdots+a_{p}\quad\mbox{and}\quad
b(z):=b_{0}+b_{1}z+\cdots+b_{p-1}z^{p-1},$$
and the
coefficients $b_{j}$ satisfy $b_{q}=1$ and $b_{j}=0$ for $q<j<p$.
The polynomials $a(\cdot)$ and $b(\cdot)$ are assumed to have no common zeroes, and it will be assumed that the zeroes of the polynomial $a$ all lie in the interior of the left half of the complex plane.

Since the derivative $DL_t$ does not exist in the usual
sense, we interpret~(\ref{1.1}) as being equivalent to the
observation and state equations
\begin{equation}\label{1.2}
Y_t=\mathbf{b}^{T}\mathbf{X}_t\,,
\end{equation}
\begin{equation}\label{1.3}
d\mathbf{X}_t=A\mathbf{X}_tdt+\mathbf{e}_{p}dL_t\,,
\end{equation}
where $$\mathbf{X}_t=\left(
          \begin{array}{c}
            X(t) \\
            X^{(1)}(t) \\
            \vdots \\
            X^{(p-2)}(t) \\
            X^{(p-1)}(t) \\
          \end{array}
        \right),\quad
       \mathbf{b}=\left(
       \begin{array}{c}
         b_{0} \\
         b_{1} \\
         \vdots \\
         b_{p-2} \\
         b_{p-1} \\
       \end{array}
     \right),\quad
     \mathbf{e}_{p}=\left(
       \begin{array}{c}
         0 \\
         0 \\
         \vdots \\
         0 \\
         1 \\
       \end{array}
     \right),$$
     $$\hskip .9in A=\left(
        \begin{array}{ccccc}
          0 & 1 & 0 & \ldots & 0 \\
          0 & 0 & 1 & \ldots & 0 \\
          \vdots & \vdots & \vdots &   \ddots & \vdots \\
          0 & 0 & 0 & \ldots & 1 \\
         -a_{p} & -a_{p-1} & -a_{p-2} & \ldots & -a_{1} \\
        \end{array}
      \right)\quad \text{and $A=-a_{1}$ for $p=1$}.$$
It is easy to check that the eigenvalues of the matrix $A$, which we shall denote by $\la_1, . . . , \la_p$, are the same as the zeroes of the autoregressive polynomial $a(\cdot)$.

Under the conditions specified it has been shown (\cite{BrLi}, Lemma~2.3) that these equations have the unique strictly stationary solution,
\begin{equation}\label{1.4}
Y_t=\int_{-\infty}^\infty g(t-u)dL_u,
\end{equation}
where
\begin{equation} \label{jordan}g(t)=\begin{cases}{\displayfrac{1}{2\pi i}}\displaystyle\int_{\rho}{\displayfrac{b(z)}{a(z)}e^{tz}dz}=\dsum_\lambda Res_{z=\lambda}\left(e^{zt}\displayfrac{b(z)}{a(z)}\right), ~&{\rm if}~t>0,\cr
                                                             0, ~&{\rm if}~ t\le 0.\end{cases}
\end{equation}
and $\rho$ is any simple closed curve in the open left half of the complex plane encircling the zeroes of $a(\cdot)$.  The sum  is over the distinct zeroes $\lambda$ of $a(\cdot)$ and $Res_{z=\lambda}(\cdot)$ denotes the residue at $\lambda$ of the function in parentheses.
Evaluating these residues, we can write $g$ more explicitly as
\begin{equation}\label{jordan:2}
g(t)=\sum_\lambda {{1}\over{(m(\lambda)-1)!}}\left[D_z^{m(\lambda)-1}
\left((z-\lambda)^{m(\lambda)}e^{zt}b(z)/a(z)\right)\right]_{z=\lambda}{\bf 1}_{(0,\infty)}(t),
\end{equation}where $m(\lambda)$ denotes the multiplicity of the zero $\lambda$ and $D_z$ denotes differentiation with respect to $z$.  The kernel $g$ can also be expressed (\cite{BrLi}, equations (2.10) and (3.7)) as
\begin{equation}\label{jordan:0}
g(t)=
\mathbf{b^\top} e^{At}\mathbf{e}_{p}{\bf 1}_{(0,\infty)}(t).
\end{equation}
From this equation we see at once that $g$ is infinitely differentiable on $(0,\infty)$ with $k^{\rm th}$ derivative,
$$g^{(k)}(t)=\mathbf{b^\top} e^{At}A^k\mathbf{e}_{p}, ~0<t<\infty.$$
Since $b_q=1$ and $b_j=0$ for $j>q$,  the right derivatives $g^{(k)}(0+)$ satisfy
\begin{equation}\label{value:D:carma}g^{(k)}(0+)=\mathbf{b^\top}A^k\mathbf{e}_{p}= \begin{cases}0&{\rm if}~ k<p-q-1,\cr
                                                                                                                1&{\rm if}~k=p-q-1,\end{cases}\end{equation}
and in particular $g(0+)=1$ if $p-q=1$ and $g(0+)=0$ if $p-q>1$.

Gaussian CARMA processes, of which the Gaussian Ornstein-Uhlenbeck process is an early example,  were first studied
in detail  by \cite{Doob:1944} (see also \cite{Doob:1990df}).  The state-space formulation, (1.2) and (1.3)  (with
$\mathbf{b^\top}=[1~0~\cdots~0]$) was
used by \cite{jones2} to carry out inference for time series with irregularly-spaced observations.  This formulation
leads naturally to the definition of L\'evy-driven and non-linear CARMA processes (see \cite{brockwell5} and the references therein).  Fractionally integrated L\'evy-driven  CARMA processes were studied by \cite{ BroMar05}.

L\'evy-driven CARMA processes  have been applied successfully to the modelling of stochastic volatility  in finance (see \cite{todorov:tauchen:2006},  \cite{BCL:2006} and  \cite{HaCz:2007}), extending the celebrated  Ornstein-Uhlenbeck model of \cite{barndorff8}.
{The results presented here were motivated by preliminary  studies of high-frequency turbulence data (see \cite{techEnzo}) which appear to be well-fitted by a continuous time moving average process sampled at times $0,\Delta,2\Delta,\ldots$, where $\Delta$ is small and positive.  We return to this topic in \cite{bfk:2011:2}}. The application to turbulence data will be investigated in detail in \cite{fk:2011:1}.


Our paper is organised as follows. In Section~\ref{s2}  we derive an expression for the spectral density of the sampled sequence  $Y^\Delta:=(Y_{n\Delta})_{n\in\bbz}$.  It is known that the filtered process $(\phi(B)Y^\Delta_n)_{n\in\bbz}$,
where $\phi(B)$ is the filter defined in (3.1),  is a moving average of order at most $p-1$.  In Section~\ref{s3}, we
determine the asymptotic behaviour of the spectral density and autocovariance function of  $(\phi(B)Y^\Delta_n)_{n\in\bbz}$ as $\Delta\downarrow 0$ and the asymptotic moving average coefficients and white noise variance in the cases $p-q=1, 2$ and $3$.  In general we show that for small enough $\Delta$ the order of the moving average $(\phi(B)Y^\Delta_n)_{n\in\bbz}$ is $p-1$.

\section{The spectral density of $Y^\Delta:=(Y_{n\Delta})_{n\in\bbz}$}\label{s2}


From \eqref{jordan} we immediately see, since $g(t)=0$ for $t<0$, that the Fourier transform of $g$ is
\begin{equation}\label{f:tran}
\tilde g(\omega):=\int_\bbr g(t)e^{i\omega t}dt
=-{{1}\over{2\pi i}}\int_\rho {{b(z)}\over{a(z)}}{{1}\over{z+i\omega}}dz
 ={{b(-i\omega)}\over{a(-i\omega)}}, \quad\omega\in\bbr.
\end{equation}
Since the autocovariance function $\gamma_Y(\cdot)$ is the convolution of $\sigma g(\cdot)$ and $\sigma g(-\cdot)$,
its Fourier transform is given by
$${\tilde\gamma}_Y(\omega)=\sigma^2{\tilde g}(\omega){\tilde g}(-\omega)=\sigma^2\left|{{b(i\omega)}\over{a(i\omega)}}\right|^2, ~~\omega\in \bbr.$$
The spectral density of $Y$ is the inverse Fourier transform of $\gamma_Y$.  Thus
$$
f_Y(\omega)=\frac{1}{2\pi}\int_\bbr  e^{-i\omega h}\gamma_Y(h) dh =\frac{1}{2\pi}{\tilde\gamma}_Y(-\omega)= \frac{\sigma^2}{2\pi} \left|\frac{b(i\omega)}{a(i\omega)}\right|^2,\quad \omega\in\bbr.
$$
Substituting this expression into the relation
$$\ga_Y(h)=\int_\bbr  e^{i\omega h}f_Y(\omega)d\omega,\quad h\in\bbr,$$
and changing the variable of integration from $\omega$ to $z=i\omega$ gives,
\begin{equation}\label{cov:sam}\gamma_Y(h)={{\sigma^2}\over{2\pi i}}\int_\rho{{b(z)b(-z)}\over{a(z)a(-z)}}e^{|h|z}dz=\sigma^2\sum_\lambda Res_{z=\lambda}\left(\frac{b(z)b(-z)}{a(z)a(-z)}e^{z|h|}\right),\end{equation}
where the sum is again over the distinct zeroes of $a(\cdot)$.


We can now compute the spectral density of the sampled sequence $Y^\Delta:=(Y_{n\Delta})_{n\in\bbz}$.  This spectral density
$f_{\Delta}$ will play a key role in the subsequent analysis. We have, from Corollary 4.3.2 in \cite{BD},
$$f_\Delta(\omega)={{1}\over{2\pi}}\sum_{h=-\infty}^\infty\gamma_Y(h\Delta)e^{-ih\omega},\quad-\pi\le\omega\le\pi,$$
and, substituting for $\gamma_Y$ from \eqref{cov:sam},
\begin{equation}\label{sampled}f_{\Delta}(\omega)={{-\sigma^2}\over{4\pi^2i}}\int_\rho {{b(z)b(-z)}\over{a(z)a(-z)}}{{\sinh(\Delta z)}\over{\cosh(\Delta z)-\cos(\omega)}}dz,\quad -\pi\le\omega\le\pi.\end{equation}

\section{The filtered sequence, $(\phi(B)Y^\Delta_n)_{n\in\bbz}$}\label{s3}

If $\lambda_1,\ldots,\lambda_p$ are the (not necessarily distinct) zeroes of $a(\cdot)$, then we know  from \cite{BrLi}, Lemma 2.1, that if we apply the filter
\begin{equation}\label{filter}
\phi(B):=\prod_{j=1}^p(1-e^{\lambda_j\Delta}B)
\end{equation}
to the sampled sequence, $Y^\Delta$,
we obtain a strictly stationary sequence which is $(p-1)$-correlated and is hence, by Lemma 3.2.1 of \cite{BD}, a moving average process of order $p-1$ or less.

Our goal in this section is to study the asymptotic properties, as $\Delta\downarrow 0$, of the moving average
$\theta(B)Z_n$ in the ARMA representation,
\begin{equation}\label{carma:sampled}
\phi(B)Y^\Delta_n=\theta(B)Z_n,\quad n\in\bbz,
\end{equation}
of the high-frequency sequence $Y^\Delta$.
Here $B$ denotes the backward shift operator and
$(Z_n)_{n\in\bbz}$ is an uncorrelated sequence of zero-mean random variables with constant variance which we shall denote by $\tau^2$.

We shall denote by $f_{MA}$ the spectral density of $(\theta(B)Z_n)_{n\in\bbz}$.  Then, observing that the power transfer function of the filter \eqref{filter} is
\begin{equation}{\label{gain}}
\psi(\omega) = |\prod_{j=1}^p (1-e^{\la_j\Delta +i\omega})|^2
=2^pe^{-a_1\Delta}\prod_{i=1}^p(\cosh(\lambda_i\Delta)-\cos(\omega)),\quad-\pi\le\omega\le\pi,\end{equation}
we have
\begin{equation}\label{filtered:spectrum}f_{MA}(\omega)=\psi(\omega)f_\Delta(\omega),\quad-\pi\le\omega\le\pi,\end{equation}
where $\psi(\omega)$ and $f_\Delta(\omega)$ are given by \eqref{gain} and \eqref{sampled} respectively.

In principle the expression \eqref{filtered:spectrum} determines the second order properties of $(\theta(B)Z_n)_{n\in\bbz}$ and in particular the autocovariances $\gamma_{MA}(h)$ for $h=0,\ldots,p-1$.  Ideally we would like to use these autocovariances to find the
coefficients $\theta_1,\ldots,\theta_{p-1}$ and white noise variance $\tau^2$, all of which are uniquely determined by the autocovariances, if we impose the condition that $\theta(\cdot)$ has no zeros in the interior of the unit circle. Determination of these quantities is equivalent to finding the corresponding factorization of the spectral density $f_{MA}$ (see \cite{Sayed} for a recent paper on spectral factorization).

From \eqref{sampled}, \eqref{gain} and \eqref{filtered:spectrum} we can calculate the spectral density $f_{MA}(\omega)$ as $-\sigma^2\psi(\omega)/(2\pi)$ times the sum of the residues in the left half plane of the integrand in \eqref{sampled} 
, i.e.
\beam\label{spectralma}
f_{MA}(\omega)= - {{\sigma^2}\over{2\pi}}\psi(\omega)\sum_\lambda D_z^{m(\lambda)-1}\left({{\sinh(\Delta z)b(z)b(-z)}\over{(\cosh(\lambda\Delta)-\cos(\omega))a(-z)\prod_{\mu\ne \lambda}(z-\mu)}^{m(\mu)}}\right)_{z=\lambda},
\eeam
where the sum is over the distinct zeroes $\lambda$ of $a(\cdot)$ and the product in the denominator is over the distinct zeroes $\mu$ of $a(\cdot)$, which are different from $\lambda$. The multiplicities of the zeroes $\lambda$ and $\mu$ are denoted by $m(\lambda)$ and $m(\mu)$ respectively.
When the zeroes $\lambda_1,\ldots,\lambda_p$ each have multiplicity 1, the expression for $f_{MA}(\omega)$ simplifies to
$$f_{MA}(\omega)={{(-2)^pe^{-a_1\Delta}\sigma^2}\over{2\pi}}\sum_{i=1}^p {{b(\lambda_i)b(-\lambda_i)}\over{a'(\lambda_i)a(-\lambda_i)}}\sinh(\lambda_i\Delta)\prod_{j\ne i}(\cos\omega-\cosh(\lambda_j\Delta)), \quad -\pi\le\omega\le\pi.$$
Although in principle the corresponding autocovariances $\gamma_{MA}(j)$ could be derived from $f_{MA}$,
we derive a more direct explicit expression later as Proposition~\ref{acf}.
The asymptotic behaviour of $f_{MA}$ as $\Delta\downarrow 0$ is derived in the following theorem by expanding \eqref{sampled} in powers of $\Delta$ and evaluating the corresponding coefficients.  Here and in all that follows
we shall use the notation,
 $a(\Delta)\sim b(\Delta)$, to mean that $\lim_{\Delta\downarrow 0} a(\Delta)/b(\Delta) =1$.

\begin{thm}\label{spectral:thm}
The spectral density $f_{MA}$ of $(\theta(B)Z_n)_{n\in\bbz}$ in the ARMA representation \eqref{carma:sampled} of the sampled process
$Y^\Delta$ has the asymptotic form, as $\Delta\downarrow 0$,
\begin{equation}\label{expansion:spectrum}f_{MA}(\omega)\sim{{\sigma^2}\over{2\pi}}(-1)^{p-q-1}\Delta^{2(p-q)-1}c_{p-q-1}(\omega)2^{p-1}(1-\cos{\omega})^p,\quad-\pi\le\omega\le\pi,\end{equation}
where $c_k(\omega)$ is the coefficient of $x^{2k+1}$ in the power series expansion
\begin{equation}\label{expansion:filter}\hskip -1.75in{{\sinh x}\over{\cosh x-\cos \omega}}
=\sum_{k=0}^\infty c_k(\omega) x^{2k+1}.\end{equation}
In particular, $c_0(\omega)={{1}\over{1-\cos\omega}},~c_1(\omega)=-{{2+\cos\omega}\over{6(1-\cos\omega)^2}},~c_2(\omega)= {{33+26\cos\omega+\cos(2\omega)}\over{240(1-\cos\omega)^3}}, ~\ldots.
$
\end{thm}
\begin{proof}
The integrand in \eqref{sampled} can be expanded as a power series in $\Delta$ using \eqref{expansion:filter}.  The integral can then be evaluated term by term using the identities, (see Example 3.1.2.3. of \cite{res:ref})
$${{1}\over{2\pi i}}\int_\rho z^{2k+1}{{b(z)b(-z)}\over{a(z)a(-z)}}dz=-{{1}\over{2}}
Res_{z=\infty}\left({{z^{2k+1}b(z)b(-z)}\over{a(z)a(-z)}}\right), \quad k\in\{0,1,2,\ldots\},$$
from which we obtain, in particular,
$${{1}\over{2\pi i}}\int_\rho z^{2k+1}{{b(z)b(-z)}\over{a(z)a(-z)}}dz=\begin{cases}0 &{\rm if}~ 0\le k<p-q-1, \cr
                                                                 {{(-1)^{p-q}}\over{2}} &{\rm if}~ k=p-q-1.
                                                                 \end{cases}$$

Substituting the resulting expansion of the integral \eqref{sampled} and the asymptotic expression  $\psi(\omega)\sim 2^p(1-\cos\omega)^p$   into \eqref{filtered:spectrum}
and retaining only the dominant power of $\Delta$ as $\Delta\rightarrow 0$, we arrive at \eqref{expansion:spectrum}.
\end{proof}

\begin{cor}
The following special cases are of particular interest.
\begin{equation}\label{cor:a}\hskip -2.3in p-q=1:~~f_{MA}(\omega)\sim{{\sigma^2\Delta}\over{2\pi}}2^{q}(1-\cos\omega)^{q}.\end{equation}
\begin{equation}\label{cor:b}\hskip -1.3in p-q=2:~~f_{MA}(\omega)\sim{{\sigma^2\Delta^3}\over{2\pi}}\left({{2}\over{3}}+{{\cos \omega}\over{3}}\right)2^{q}(1-\cos\omega)^{q}.\end{equation}
\begin{equation}\label{cor:c}p-q=3:~~f_{MA}(\omega)\sim{{\sigma^2\Delta^5}\over{2\pi}}\left ({{11}\over{20}}+{{13\cos \omega}\over{30}}+{{\cos(2\omega)}\over{60}}\right) 2^{q}(1-\cos\omega)^{q}.\end{equation}
\end{cor}
\begin{proof}
These expressions are obtained from \eqref{expansion:spectrum} using the values of $c_0(\omega), c_1(\omega)$ and $c_2(\omega)$ given in the statement of the theorem.
\end{proof}

\brem
(i) \,
The right-hand side of \eqref{cor:a} is the spectral density of a $q$-times differenced white noise with variance $\sigma^2\Delta$.
It follows that, if $q=p-1$, then the moving average polynomial $\theta(B)$
in \eqref{carma:sampled} is asymptotically $(1-B)^q$ and the white noise variance $\tau^2$ is asymptotically $\sigma^2\Delta$ as $\Delta\rightarrow 0$. This result is stated with the corresponding results for $p-q=2$ and $p-q=3$  in the following corollary.\\[2mm]
(ii) \, By Proposition~3.32 of \cite{MaSt:2007} a CARMA($p,q$)-process has sample paths which are $(p-q-1)$-times differentiable.
Consequently to represent processes with non-differentiable sample-paths it is necessary to restrict attention to the case $p-q=1$.  It is widely believed that sample-paths with more than two derivatives are too smooth to represent the processes observed empirically in finance and turbulence (see e.g. \cite{JaTo:2010,JKM:2010}) so we are not concerned with the cases when $p-q>3$.
\erem
\begin{cor}\label{spectralcor}
The moving average process $X_n:=\theta(B)Z_n$ in \eqref{carma:sampled} has for $\Delta\downarrow 0$ the following asymptotic form.\\[2mm]
(a) If  $p-q=1$, then
$$X_n=(1-B)^{q}Z_n,\quad n\in\bbz,$$
where $\tau^2:=Var(Z_n)=\sigma^2\Delta$.\\[2mm]
(b) If $p-q=2$, then
$$X_n=(1+\theta B)(1-B)^{q}Z_n,\quad n\in\bbz,$$
where $\theta=2-{\sqrt{3}}$ and $\tau^2:=Var(Z_n)=\sigma^2\Delta^3(2+\sqrt{3})/6$.\\[2mm]
(c) If $p-q=3$, then
$$X_n=(1+\theta_1 B+\theta_2 B^2)(1-B)^{q}Z_n,\quad n\in\bbz,$$
where $\theta_2=2 \left(8+\sqrt{30}\right)-\sqrt{375+64 \sqrt{30}}$, $\theta_1=26\theta_2/(1+\theta_2)=13-\sqrt{135+4 \sqrt{30}}$ and
$\tau^2=\left(2 \left(8+\sqrt{30}\right)+\sqrt{375+64 \sqrt{30}}\right) \Delta ^5 \sigma ^2/120$.
\end{cor}

\begin{proof} (a) follows immediately from Theorem 4.4.2 of \cite{BD}.

To establish (b) we observe from \eqref{cor:b} that the required moving average is the $q$ times differenced MA(1) process with autocovariances at lags zero and one, $\gamma(0)=2\sigma^2\Delta^3/3$ and
$\gamma(1)=\sigma^2\Delta^3/6$.  Expressing these covariances in terms of $\theta$ and $\tau^2$ gives the equations,
$$(1+\theta^2)\tau^2=2\sigma^2\Delta^3/3,$$
$$\theta\tau^2=\sigma^2\Delta^3/6,$$
from which we obtain a quadratic equation for $\theta$. Choosing the unique solution which makes the MA(1) process
invertible gives the required result.

The proof of (c) is analogous.  The corresponding argument yields a quartic equation for $\theta_2$.  The particular solution given in the statement of (b) is the one which satisfies the condition that $\theta(z)$ is nonzero for all complex $z$ such that $|z|<1$.
\end{proof}

{Although the absence of the moving-average coefficients, $b_j$, from Corollary 3.4 suggests that they cannot be estimated from very closely-spaced observations, the coefficients do appear if the expansions are taken to higher order in $\Delta$.  The apparent weak dependence of the sampled sequence on the moving-average coefficients as $\Delta\downarrow 0$ is compensated by the increasing number of available observations.
}

In principle the autocovariance function $\gamma_{MA}$ can be calculated, as indicated earlier, from the corresponding spectral density $f_{MA}$ given by \eqref{filtered:spectrum} and \eqref{f:tran}.
Below we derive a more direct representation of $\gamma_{MA}$ and use it to prove Theorem~\ref{acv:approx:thm}, which is the time-domain analogue of Theorem~\ref{spectral:thm}.

Define $B_\Delta g(t)=g(t-\Delta)$ for $t\in\bbr$.
We show that $\phi(B_\Delta)g(\cdot)\equiv 0$ for $t>p\Delta$.

\begin{lem}\label{pq}
Let $Y$ be the $\CARMA(p,q)$ process \eqref{1.4} and $\Delta>0$.
Define $\phi(B)$ as in \eqref{filter}. Then
\beam\label{prod}
\phi(B_\Delta)g(t):=\prod_{j=1}^p(1-e^{\lambda_j\Delta}B_\Delta)g(t)=0, \quad t>p\Delta.
\eeam
\end{lem}

\begin{proof}
Rewriting the product in \eqref{prod} as a sum we find
$\phi(B_\Delta)g(t) = \sum_{j=0}^p A_j^p g(t-j\Delta)$,
which has Fourier transform (invoking the shift property
and the right hand side of \eqref{f:tran})
$$\prod_{\lambda}(1-e^{\Delta(\lambda+i\omega)})^{m(\lambda)}\frac{b(-i\omega)}{a(-i\omega)},\quad \omega\in\bbr,$$
where the product is taken over the distinct zeroes of $a(\cdot)$ having multiplicity $m(\lambda)$.
Using the fact that the product of Fourier transforms corresponds to the convolution of functions, we obtain from \eqref{jordan}
$$\phi(B_\Delta)g(t)=-\frac{1}{2\pi i}\int_\rho \prod_{\lambda}(1-e^{\Delta(\lambda-z)})^{m(\lambda)} \frac{b(z)}{a(z)}e^{tz}dz
=-\sum_{\lambda}Res_{z=\lambda}\left(e^{zt}b(z)\prod_{\lambda}\frac{(1-e^{\Delta(\lambda-z)})^{m(\lambda)}}{(z-\lambda)^{m(\lambda)}}\right).$$
Now note that, for every of the distinct zeroes $\lambda_j$,
$$\lim_{z\rightarrow\lambda_i}\frac{(1-e^{\Delta(\lambda_j-z)})^{m(\lambda_j)}}{(z-\lambda_j)^{m(\lambda_j)}}=\Delta^{m(\lambda_j)}.$$
The singularities at $z=\lambda_j$ are removable and, therefore, using Cauchy's residue theorem, {Theorem 1} of Section~3.1.1, p.~25, and Theorem~2 of Section~2.1.2, p.~7, of \cite{res:ref}, the filtered kernel is zero for every $t\in\bbr$.
\end{proof}

\begin{prop}\label{acf}
Let $Y$ be the $\CARMA(p,q)$ process \eqref{1.4} and $\Delta>0$.
The autocovariance at lag n of $(\phi(B)Y^\Delta_j)_{j\in\bbz}$ is, for $n=0,1,\ldots,p-1$,
\begin{equation}
\label{simp:acv}\gamma_{{MA}}(n)=\sigma^2\sum_{i=1}^{p-n}~\sum_{k=0}^{n+i-1}\sum_{h=0}^{i-1}A^p_kA^p_h
\int_{(i-1)\Delta}^{i\Delta}g(s-h\Delta)g(s-(k-n)\Delta)ds,
\end{equation}
with
\begin{equation}\label{coeffs}
A^p_k=(-1)^k\sum_{\{i_1,\ldots,i_k\}\in C_k^p}e^{\Delta(\lambda_{i_1}+\cdots+\lambda_{i_k})},\quad k=1,\ldots,p.
\end{equation}
The sum in \eqref{coeffs} is taken over the $p\choose k$ subsets of size $k$ of $\{1,2,\ldots,p\}$.
\end{prop}

\begin{proof}
We note that $\gamma_{{MA}}(n)$ is the same as  $\bbe[(\phi(B_\Delta)Y)_t(\phi(B_\Delta)Y)_{t+\Delta n}]$                          and use the same expansion as in the proof of Lemma~\ref{pq}, i.e.
\begin{equation}\label{filt:expansion}
\phi(B_\Delta)=\prod_{j=1}^p(1-e^{\lambda_j\Delta}B_\Delta)
= \sum_{k=0}^pA^p_kB_\Delta^k,
\end{equation}
which we apply to $Y$.
Observe that for $t\in\bbr$, setting $t_k:=t-k\Delta$ for $k=0,\ldots,p$, and $t_{p+1}:=-\infty$,
\begin{equation}\label{es}
B_\Delta^k Y_t=B_\Delta^k\int_{-\infty}^tg(t-u)dL_u=\int_{-\infty}^{t_k}g(t_k-u)dL_u=\sum_{i=k}^p\int_{t_{i+1}}^{t_i}g(t_k-u)dL_u.
\end{equation}
Applying the operator \eqref{filt:expansion} to $Y_t$, using \eqref{es} and interchanging the order of summation gives
\begin{equation}\label{I:j}
(\phi(B_\Delta)Y)_t=\sum_{m=0}^p\int_{t_{m+1}}^{t_m}\sum_{k=0}^mA_k^pg(t_k-u)dL_u.
\end{equation}
From Lemma~\ref{pq} we know that the contribution from the term corresponding to $m=p$ is zero.
By stationarity, the autocovariance function is independent of $t$, hence we can choose $t=\Delta n$.
Then we obtain
\beao
(\phi(B_\Delta)Y)_{n\Delta}
&=& \sum_{j=0}^{p-1}\int_{\Delta(n-j-1)}^{\Delta (n-j)}\sum_{k=0}^j A_k^pg((n-k)\Delta -u)dL_u\\
 &=& \sum_{j=1}^{p}\int_{\Delta(n-j)}^{\Delta (n-j+1)}\sum_{k=0}^{j-1} A_k^pg((n-k)\Delta -u)dL_u.
\eeao
For $t=0$, we obtain analogously
\beao
(\phi(B_\Delta)Y)_{0}
=\sum_{i=1}^{p}\int_{-\Delta i}^{-\Delta (i-1)}\sum_{h=0}^{i-1} A_h^p g(-\Delta h -u)dL_u.
\eeao
For the autocovariance function we obtain for $n=0,\ldots,p-1$ by using the fact that $L$ has orthogonal increments,
\beao\label{A:5}
\gamma_{MA}(n) &=& \bbe[(\phi(B_\Delta)Y)_0(\phi(B_\Delta)Y)_{n\Delta}]\\
&=& \si^2  \sum_{i=1}^{p-n}  \sum_{k=0}^{i+n-1} \sum_{h=0}^{i-1} A_k^p A_h^p
\int_{-i\Delta}^{-(i-1)\Delta}   g((n-k)\Delta -u)   g(-\Delta h -u) du.\\
\eeao
Finally, \eqref{simp:acv} is obtained by changing the variable of integration from $u$ to $s=-u$.
\end{proof}

\begin{thm}\label{acv:approx:thm}
The autocovariance function $\gamma_{MA}(n)$ for $n=1,\ldots,p-1$ has for $\Delta\downarrow 0$ the asymptotic form
\begin{equation}\label{acv:approx}
\gamma_{{MA}}(n) \sim \frac{\sigma^2 \Delta^{2(p-q)-1} }{((p-q-1)!)^2}\sum_{i=1}^{p-n}~\sum_{k=0}^{n+i-1}\sum_{h=0}^{i-1} (-1)^{h+k}{p \choose k}{p \choose h}C(h,k,i,n;p-q-1),\end{equation}
where for $N\in\bbn_0$
\beao
C(h,k,i,n;N) &:=& \int_{0}^{1}(s + i - 1-h)^{N}(s + i - 1-k+n)^{N}ds.
\eeao
\end{thm}

\begin{proof}
We can rewrite the integral in (\ref{simp:acv}) as
\beam\label{deltaint}
\Delta\int_{0}^{1}g((s + i - 1-h)\Delta)g((s + i - 1-k+n)\Delta)ds.
\eeam
Since $g$ is infinitely differentiable on $(0,\infty)$ and the right derivatives at $0$ exist, the integrand  has  one-sided Taylor expansions of all orders $M\in\bbn$,
\beao
&&\sum_{l=0}^M  \left.\frac{d^{l}\left[g((s + i - 1-h)\Delta)g((s + i - 1-k+n)\Delta)\right]}{d\Delta^l}\right|_{\Delta=0^{+}}\frac{\Delta^l}{l!}+o\left(\Delta^M\right)\\
&&=\sum_{l=0}^M\sum_{m=0}^l {l \choose m}(s + i - 1-h)^{l-m}(s + i - 1-k+n)^{m}g^{(l-m)}(0+)g^{(m)}(0+)\frac{\Delta^l}{l!}+o\left(\Delta^M\right),
\eeao
as $\Delta\downarrow 0$.
Choose $M=2(p-q-1)$.  Then by \eqref{value:D:carma} there is only one term in the double sum which does not vanish, namely the term for which $m=p-q-1=l-m$.  Setting $N:=p-q-1$ (so that $M=2N$) the sum reduces to
$$  {2N \choose N}(s + i - 1-h)^{N}(s + i - 1-k+n)^{N}\frac{1}{(2N)!}\Delta^{2N}+o\left(\Delta^{2N}\right).$$
Since ${2N\choose N}/(2N)!=(N!)^{-2}$, the integral in \eqref{deltaint} is for $\Delta\downarrow 0$ asymptotically equal to
\begin{equation}\label{exp:1}
\frac{\Delta^{2N+1}}{\left({N!}\right)^{2}}\int_{0}^{1}(s + i - 1-h)^{N}(s + i - 1-k+n)^{N}ds+o(\Delta^{2N+1}),
\end{equation}
and, since
$$
\lim_{\Delta\downarrow 0}\sum_{\left\{i_1,\ldots,i_h\right\} \in C_{h}^{p}}e^{\Delta(\lambda_{i_1}+\cdots +\lambda_{i_h})}
={p\choose h},
$$
we also have
\begin{equation}\label{exp:2}
A^p_kA^p_h=(-1)^{h+k}{p \choose k}{p \choose h}+o(1)\quad {\rm as}~\Delta\downarrow 0.
\end{equation}
Combining  (\ref{exp:1}) and (\ref{exp:2}), we obtain (\ref{acv:approx}).
\end{proof}

\brem \label{remark:calc}
(i) For computations the following expansion may be useful (as usual we set $0^0=1)$
\beao
C(h,k,i,n;N)&:=&\int_{0}^{1}(s + i - 1-h)^{N}(s + i - 1-k+n)^{N}ds\\
&=&\sum_{l_1,l_2=0}^N{N\choose l_1}{N\choose l_2}  (i - 1-h)^{N-l_1}(i - 1-k+n)^{N-l_2} \int_{0}^1{s^{l_1+l_2}}ds \\
&=&\sum_{l_1,l_2=0}^N{N\choose l_1}{N\choose l_2}\frac{1}{l_1+l_2+1}(i - 1-h)^{N-l_1}(i - 1-k+n)^{N-l_2}.
\eeao
Furthermore, we observe that $C$ depends on $p$ and $q$ only through $p-q$.\\
(ii)  Note that the right hand sides of \eqref{spectralma} and \eqref{expansion:spectrum} are the discrete Fourier transforms of  \eqref{simp:acv} and \eqref{A:5}, respectively.
Note also the symmetry between \eqref{spectralma} and \eqref{simp:acv} in the dependence on $\Delta$ and $p-q-1$.
\erem

So far we know that the moving average process $X_n=\theta(B)Z_n$ from \eqref{carma:sampled} is of order not greater than $p-1$ but possibly  lower. Our next result presents an asymptotic formula for $\ga_{MA}(p-1)$, which shows clearly that this term is not 0.

\begin{cor}\label{corol:acv}
For lag $n=p-1$ the autocovariance formula
\eqref{acv:approx} reduces to
\begin{equation}\label{p:2}\gamma_{{MA}}(p-1) \sim (-1)^{q}\frac{\sigma^2
\Delta^{2(p-q)-1} }{(2(p-q-1))!}
\end{equation}
and $\gamma_{MA}(p-1)$ is therefore non-zero for all sufficiently small $\Delta>0$.
\end{cor}

\begin{proof}
From the expansion \eqref{acv:approx} we find
\begin{equation}\label{acv:p-1}\gamma_{{MA}}(p-1) \sim \frac{\sigma^2
\Delta^{2(p-q)-1} }{((p-q-1)!)^2}\sum_{k=0}^{p-1} (-1)^{k}{p \choose
k}C(0,k,1,p-1;p-q-1).
\end{equation}
Set $d:=p-q\geq 1$, then
$$C(0,k,1,p-1;d-1)=\int_{0}^{1}s^{d-1}(s-k+p-1)^{d-1}ds,$$
and, from Remark \ref{remark:calc}, this is a polynomial of
order $d-1$.
In order to apply known results on the difference operator, we define
the polynomial $f(x)=\int_0^1 s^{d-1}(x+s+p-1)^{d-1} ds$.
Then, using Eq. (5.40), p.~188, and the last formula on p.~189 in
\cite{knuth}, the sum in \eqref{acv:p-1} can be written as
\beam
&& \sum_{k=0}^{p-1} (-1)^{k}{p \choose k}C(0,k,1,p-1;d-1)\nonumber\\
&=& \sum_{k=0}^{p} (-1)^{k}{p \choose k} f(x-k)|_{x=0}
-(-1)^p{p \choose p} C(0,k,1,p-1;d-1) \nonumber\\
&=& 0+(-1)^{p+1}\int_{0}^1s^{d-1}(s-1)^{d-1}ds
\, = \, (-1)^{p+d}\int_{0}^1s^{d-1}(1-s)^{d-1}ds\label{integral:acv:p-1},
\eeam
where we have used the fact that $d-1=p-q-1<p$.
To obtain Eq.~\eqref{p:2} it suffices to note that $(-1)^{p+d}=(-1)^{2p-q}=(-1)^{q}$ and that the integral in \eqref{integral:acv:p-1} is a beta function. Hence
$$\int_0^1 s^{d-1}(1-s)^{d-1}ds = \frac{(\Gamma(d))^2}{ \Gamma(2d)}= \frac{((d-1)!)^2}{(2d-1)!}>0,\quad d\in\bbn.$$
\end{proof}
\brem If $Y$ is the $\CARMA(p,q)$ process (1.4) then, from Theorem \ref{spectral:thm}, the spectral density of $(1-B)^{p-q}Y^\Delta$ is asymptotically, as $\Delta\downarrow 0$,
   $$\frac{\sigma^2}{2\pi}\Delta(-2\Delta^2)^{p-q-1}c_{p-q-1}(\omega)(1-\cos\omega)^{p-q},\quad \pi\leq \omega\leq \pi.$$
If $p-q=1$, $2$ or $3$ this reduces to the corresponding spectral densities in Corollary 3.2, each  divided by $2^q(1-\cos\omega)^q$.  The corresponding  moving average representations are as in Corollary 3.4 without the factors $(1-B)^q$ .

In particular, for the $\CAR(1)$ process, $(1-B)Y^\Delta$ has a spectral density which is asymptotically $\sigma^2\Delta/(2\pi)$
so that, in the Gaussian case, the increments of  $Y^\Delta$ for small $\Delta$ approximate those of Brownian motion
with variance $\sigma^2 t$.
\erem
{
In this paper we have considered only second-order properties of $Y^\Delta$.  It is possible  (see \cite{brockwell5}, Theorem 2.2) to express the joint characteristic functions, $\bbe\exp( i\sum_{k=1}^m\theta_kY^\Delta_k)$, for $m\in\bbn,$ in terms of the coefficients
 $a_j$ and $b_j$ and the function $\xi(\cdot)$, where $\xi(\theta)$ for  $\theta\in\mathbb{R}$ is the exponent in the characteristic function
$\bbe e^{i\theta L_1}=e^{\xi(\theta)}$ of $L_1$.  In particular the marginal characteristic function is given by $\bbe \exp(i\theta Y^\Delta_k)= \exp \int_0^\infty \xi(\theta{\bf b}'e^{Au} {\bf e}) du.$

These expressions are awkward to use in practice, however \cite{BDY3} have found that least squares estimation (which depends only on second-order properties) for closely and uniformly spaced observations of a CARMA(2,1) process on a fixed interval $[0,T]$ gives good results.  They find in simulations that for large $T$ the empirically-determined sample covariance matrix of the estimators of $a_1,a_2$ and $b_0$ is close to the matrix calculated from the asymptotic (as $T\rightarrow\infty$) covariance matrix of the maximum likelihood estimators based on continuous observation on $[0,T]$ of the corresponding Gaussian CARMA process.}

\begin{table}
\begin{tabular}{|c|c|c|c|c|c|c|}
\hline
$p-q$ & 1&2&3&4\\
\hline
$\gamma_{MA}(p-1)$&$\Delta  (-1)^{p-1} \sigma ^2$ &
$ 6^{-1}\Delta ^3 (-1)^{p-2} \sigma ^2$ &
$ 120^{-1} \Delta ^5 (-1)^{p-3} \sigma ^2$ &
$ {5040}^{-1}{\Delta ^7 (-1)^{p-4} \sigma ^2}$ \\[1mm]
\hline
\end{tabular}
\caption{Values of $\gamma_{MA}(p-1)$ for $p-q=1,\ldots,4$.}
\end{table}

\section{Conclusions}

When a CARMA($p,q$) process $Y$ is sampled at times $n\Delta$ for $n\in\bbz$, it is well-known that
the sampled process $Y^\Delta$ satisfies discrete-time ARMA equations of the form \eqref{carma:sampled}.
The determination
of the moving average coefficients and white noise variance for given grid size $\Delta$, however, is a non-trivial procedure.
In this paper we have focussed on {\em high frequency sampling} of $Y$.
We have determined the relevant second order quantities,
 the spectral density $f_{MA}$ of the moving average on the right-hand side of \eqref{carma:sampled}
and its asymptotic representation as $\Delta\downarrow 0$.
This includes the moving average coefficients as well as the variance of the innovations.
 We also derived an explicit expression for the autocovariance function $\gamma_{MA}$ and its asymptotic representation as $\Delta\downarrow 0$.
 This shows, in particular,  that the moving average is of order $p-1$ for $\Delta$ sufficiently small.

\section{Acknowledgments}
PJB gratefully acknowledges the support of this work by NSF Grant DMS 0744058 and the Institute of Advanced Studies at Technische Universit\"at M\"unchen where this work was initiated. 
The work of Vincenzo Ferrazzano was supported by the International Graduate School of Science and Engineering (IGSSE) of Technische Universit\"at M\"unchen.
We are also indebted to a referee for valuable comments.

\bibliographystyle{kluwer}
\bibliography{bib}
\end{document}